\documentclass[reqno,10pt]{amsart}
\usepackage{amsmath,amsthm,amsfonts,amssymb}

\date{}
\newtheorem{theorem}{Theorem}[section] \newtheorem{lemma}[theorem]{Lemma}
 \newtheorem{remark}[theorem]{Remark}
\newtheorem{proposition}[theorem]{Proposition}

\newtheorem{example}[theorem]{Example} \numberwithin{equation}{section}

\begin{document}

\vspace*{0.2cm}

\centerline{\Large\bf On the well-posedness of the Cauchy problem}

\vskip .1in

\centerline{\Large\bf for Fokker--Planck--Kolmogorov equations}

\vskip .1in

\centerline{\Large\bf with potential terms on  arbitrary  domains}

\vspace*{0.1in}

\quad {\sc OXANA A. MANITA}${}^{1}$,

\quad {\sc STANISLAV V.~SHAPOSHNIKOV}${}^{2}$

\vspace*{0.1in}

\quad ${}^{1}$ Department of Mechanics and Mathematics, Moscow State
University, 119991 Moscow, Russia

\quad ${}^{2}$ Department of Mechanics and Mathematics, Moscow State
University, 119991 Moscow, Russia

\vskip 0.3in

{\it We study the Cauchy problem for Fokker--Planck--Kolmogorov equations with
unbounded and degenerate coefficients. Sufficient conditions for the existence and
uniqueness of solutions are indicated. }

\vskip 0.2in


\vskip 0.2in


\section{\sc Introduction}

In this paper we study the Cauchy problem for the Fokker--Planck--Kolmogorov
equation \begin{equation}\label{e1}
\partial_t\mu=\partial_{x_i}\partial_{x_j}\bigl(a^{ij}\mu\bigr)-
\partial_{x_i}\bigl(b^i\mu\bigr)+c\mu, \quad \mu|_{t=0}=\nu. \end{equation}
Throughout the paper summation over all repeated indices is meant.

Equations of this type  for transition probabilities of diffusion processes
were first derived by Kolmogorov in his famous paper \cite{K}. In the same paper,
the question  about the existence and uniqueness
of  probability solutions was posed (the case $c=0$).
The classical works \cite{Ar, Fr, Str-V, Tihonov, Vidder} deal with such equations with smooth
coefficients, having at most linear growth at infinity.

Equations with integrable and Sobolev coefficients in the class of bounded Borel
measures have been intensively studied in the past decade. For the
variational approach to (\ref{e1}) in the case of
unit diffusion matrix, a gradient drift and $c=0$, see \cite{JKO}.
The existence and uniqueness of  solutions given by flows of probability measures
in the case where $c=0$, the  diffusion matrix is  nondegenerate and Sobolev regular and
the drift  is integrable have been
studied in  \cite{BDR, BDRST, BKR-Servy, BRSH-Servy}. The papers
\cite{Figalli, LeBrL, RZ} deal with  equations with degenerate diffusion matrices.
In particular,  the solvability of the Cauchy problem for
equations with a degenerate Sobolev  regular $A$ in the class of densities under
certain growth restrictions on the lower order terms has been proved in   \cite{LeBrL}.
Relations
between the $L^1$- and $L^{\infty}$-uniqueness of semigroups, Liouville-type  theorems
and the uniqueness of the $L^1$-solution to the Cauchy problem for the Fokker--Planck--Kolmogorov
equation have been studied in  \cite{Ldm, Wang}.

We note that in the papers mentioned above only equations on all of $\mathbb{R}^d$ have been
considered. However,
 the problem of existence and uniqueness of solutions to the
Cauchy problem for the Fokker--Planck--Kolmogorov equation
with irregular coefficients on an arbitrary domain $D$ is also
of great interest.
For example in  \cite{K-G}, diffusion processes on an
arbitrary domain $D\subset\mathbb{R}^d$ were studied, in particular, the process
on  $D=(-1, 1)$ with generator  $$ Lu(x)=2^{-1}\bigl|1-
|x|\bigr|^{\alpha}u''(x)+\bigl({\rm tg }(-\pi x/2)+{\rm sgn} x\bigr)u'(x), \quad
\alpha>0. $$

The main difference between our results and the known ones is that we
consider equations with potential terms on arbitrary domains with an
arbitrary  probability measure as the initial data, but even for $c=0$ and $D=\mathbb{R}^d$ 
our results are new. We extend the sufficient conditions for the existence
of solutions obtained in \cite{BDR} to the case of nondegenerate
equations without restrictions on the smoothness of the diffusion matrix under the
assumption that the drift and potential are locally bounded. Also, we impose no
global restrictions on the coefficients to prove the existence. In the
case $c=0$, our method of constructing solutions differs from
the ones used in the papers mentioned above, namely, first we construct a
subprobability solution (this step is usually much easier) and then employ
a Lyapunov function to ensure that the constructed solution is a probability
solution.

Our uniqueness results are mostly extensions of the results in
\cite{BRSH-Servy} to the case of  equations with potentional terms on
arbitrary domains. However, an important difference is that
we have managed to eliminate the assumption (which was crucial
in the paper cited) that the Lyapunov functions involved 
are globally Lipschitzian.

To be more specific, we prove that under rather broad assumptions about
coefficients  $a^{ij}$, $b^i$, and $c$, the existence of a Lyapunov function $V$
(i.e., $V\in C^2(D)$ and $V(x)\to +\infty$ as $x\to \partial D$)
such that
$$ a^{ij}(x, t)\partial_{x_i}\partial_{x_j}V(x)+b^i(x, t)\partial_{x_i}V(x)+c(x, t)V(x)\le K+KV(x), \quad K>0 $$
ensures both the existence and
uniqueness of a solution to the  Cauchy problem (\ref{e1}) given by a flow of
subprobability measures $\mu_t$ such that the identity
$$
\mu_t(D)=\nu(D)+\int_0^t\int_{D}c(x, s)\,d\mu_s\,ds $$
holds. If $c=0$, then the measures $\mu_t$ are probability measures. In particular, if $D=\mathbb{R}^d$ and
$V(x)=|x|^2/2$, then for the existence and uniqueness it is enough to have the inequality
$$
{\rm tr}A(x, t)+(b(x, t), x)+|x|^2c(x, t)/2\le K+K|x|^2.
$$  

Thus, in the case of equations with unbounded coefficients we give an answer to the
question about the existence and uniqueness of solutions to the Cauchy problem
(\ref{e1}) posed by Kolmogorov in  \cite{K}.

We now proceed to the definitions and exact statements.

Let $T>0$ and let $D$ be an arbitrary open set in $\mathbb{R}^d$. We assume that
along with the domain $D$ an increasing sequence of bounded open sets $D_k$
is given such that  for every $k$ the closure  $\overline{D_k}$ of $D_k$ belongs to
$D_{k+1}$ and $\bigcup_{k=1}^{\infty}D_k=D$. For example, if $D=\mathbb{R}^d$,
then for  $D_k$ the ball of radius $k$ centered at the origin can be taken.

We shall say that a locally finite Borel measure $\mu$ on the strip $D\times(0,
T)$ is given by a flow of Borel measures  $(\mu_t)_{t\in (0, T)}$ if, for every
Borel set $B\subset D$, the mapping $t\mapsto \mu_t(B)$ is measurable and for every
function  $u\in C^{\infty}_0(D\times(0, T))$ one has
$$
\int_{D\times(0, T)}u(x, t)\,d\mu=\int_0^T\int_{D}u(x, t)\,d\mu_t\,dt.
$$
Obviously, the last identity
extends to all functions of the form  $fu$, where $u$ is as before and  $f$ is
$\mu$-integrable on every compact set in $D\times(0, T)$. For example, the transition
probabilities $\mu_t(B)=P(x_t\in B)$ of a stochastic process $x_t$ in
$D$ define a measure $\mu=\mu_t\,dt$ on $D\times(0, T)$.

Set
$$
L\varphi=a^{ij}\partial_{x_i}\partial_{x_j}\varphi+b^i\partial_{x_i}\varphi+c\varphi,
$$
where $a^{ij}$, $b^i$, $c$ are Borel functions on  $D\times[0, T]$ and
$A=(a^{ij})$ is a symmetric non-negative definite matrix (called the diffusion
matrix), i.e.,   $a^{ij}=a^{ji}$,
$(A(x, t)y, y)\ge 0$ for all  $(x, t)\in D\times[0, T]$ and all
$y\in\mathbb{R}^d$. The mapping  $b$ is called the drift coefficient and $c$ is called the potential.

We shall say that a measure $\mu=(\mu_t)_{t\in(0, T)}$ satisfies the  Cauchy
problem  (\ref{e1}) if $a^{ij}$, $b^i$ and $c$ belong to
$L^1(\overline{D_k}\times J, |\mu|)$ for each domain $D_k$ and each interval
$J\subset(0, T)$ and for every function  $\varphi\in C^{\infty}_0(D)$ the
following identity holds:
\begin{equation}\label{ee1}
\int_{D} \varphi d\mu
_{t}-\int_{D} \varphi d\nu =\lim_{\varepsilon\to 0+}
\int_{\varepsilon}^{t}\int_{D} L\varphi d\mu _{s}ds
\end{equation}
for a.e. $t\in(0, T)$. We note that in general the set of points  $t$ for which the identity
(\ref{ee1}) holds depends on  $\varphi$. If the function
$\displaystyle t\to\int_{D}\varphi\,d\mu_t$ is continuous on $(0, T)$, then identity
(\ref{ee1}) holds for all  $t\in[0, T)$. If one has the inclusion
$L\varphi\in L^1(D\times[0, T])$, then
$$
\lim_{\varepsilon\to 0+}
\int_{\varepsilon}^{t}\int_{D} L\varphi d\mu _{s}ds=
\int_{0}^{t}\int_{D}L\varphi d\mu _{s}ds.
$$

We shall also use another definition of a solution, which is, however, equivalent
to the previous one (see \cite{BRSH-Servy}). Namely, the measure
$\mu=(\mu_t)_{0<t<T}$ satisfies  $\partial_t\mu=L^{*}\mu$ if
$$
\int_0^T\int_{D}\bigl[\partial_tu+Lu\bigr]\,d\mu_t\,dt=0 \quad \forall u\in
C^{\infty}_0(D\times(0, T)).
$$
The measure $\mu=(\mu_t)_{0<t<T}$ satisfies the
initial condition $\mu|_{t=0}=\nu$ if, for each function  $\varphi\in
C^{\infty}_0(D)$, there exists a full Lebesgue measure set $J_{\varphi}\subset(0,
T)$ such that
$$ \lim_{t\to 0, t\in
J_{\varphi}}\int_{D}\varphi\,d\mu_t=\int_{D}\varphi\,d\nu.
$$
As before, if the
function  $\displaystyle t\to\int_{D}\varphi\,d\mu_t$ is continuous on  $(0, T)$, then
$J_{\varphi}=(0, T)$.

We shall  always assume that $c\le 0$. This assumption can
obviously be replaced with  $c\le c_0$ for some number  $c_0$.
Indeed, in order to remove $c_0$ it suffices to consider $e^{-c_0t}\mu_t$ in place of $\mu_t$.

We study the existence and uniqueness in the class
$\mathcal{M}_{\nu}$ of measures  $\mu$ given by flows of such nonnegative measures
$(\mu_t)_{0<t<T}$ that  $\mu$ is a solution to the  Cauchy problem (\ref{e1}),
$|c|\in L^1(D\times(0, T), \mu)$ and for a.e.  $t\in(0, T)$ the inequality
\begin{equation}\label{con1}
\mu_t(D)\le
\nu(D)+\int_0^t\int_{D}c(x, s)\,\mu_s(dx)\,ds
\end{equation}
holds.

The main goal of this paper is to obtain sufficient conditions, under which the
set  $\mathcal{M}_{\nu}$ consists of exactly one element. Moreover, we are interested in
conditions that admit unbounded and degenerate coefficients.

\section{\sc Existence results}

In the present section we prove several existence theorems under different assumptions
about the coefficients.

\begin{theorem}\label{th1-1}
Suppose that  $c\le0$,  for each $k\in\mathbb{N}$
the coefficients $a^{ij},\, b^{i}$ and $c$ are bounded on $D_k\times[0, T]$ and
there exist  positive numbers ~$m_k$ and $M_k$ such that the inequality
$$
m_k|y|^{2}\le (A(x, t)y, y)\le M_k|y|^{2}
$$
holds for all  $y\in\mathbb{R}^d$ and
$(x, t)\in D_k\times[0, T]$. Then,  for every probability measure  $\nu$, the set
$\mathcal{M}_{\nu}$ is not empty.
\end{theorem}
\begin{proof} We divide the proof into several steps.

1. We set  $a^{ij}(x, t)=0$, $b^{i}(x, t)=0$ and $c(x, t)=0$ if
$t\not\in [0, T]$ or if  $t\in[0, T]$, but  $x\not\in D$. Let $\omega$ be a
homogenization kernel, i.e.,
$$ \omega\in C^{\infty}_0(\mathbb{R}^{d+1}), \quad
\omega\ge 0, \quad \int_{\mathbb{R}^{d+1}}\omega(x, t)\,dx\,dt=1.
$$
Set
$\omega_{\varepsilon}(x, t)=\varepsilon^{-d-1}\omega(x\varepsilon^{-1},
t\varepsilon^{-1})$.
Let $I_n$ be the indicator of $D_n\times[0, T]$. Let
$$
a^{ij}_{n}=(a^{ij}I_n+\delta^{ij}(1-I_n))*\omega_{1/n}, \quad
b^i_n=(b^iI_n)*\omega_{1/n}, \quad {\rm and} \quad c_n=(cI_n)*\omega_{1/n}.
$$
It is quite obvious
that for every fixed  $n$ the functions $a^{ij}_n$, $b^i_n$ and $c_n$ are
smooth and uniformly bounded together with all their derivatives. Moreover,
$(A_n(x, t)y, y)\ge |y|^2\min\{m_n, 1\}$ for all  $t$, $x$ and $y$, here $m_n$
is a number from the assumptions of the theorem corresponding to the set
$D_n\times[0, T]$. Finally,  for every $k$ and $p\ge 1$ the sequences
$a^{ij}_n$, $b^i_n$, $c_n$ converge in  $L^p(D_k\times[0, T])$ to the functions
$a^{ij}$, $b^i$ and $c$ respectively.

We extend  the measure $\nu$ by zero outside  $D$ to a measure on $\mathbb{R}^d$.
Let  $\eta_{n}\in C_{0}^{\infty}(D)$ be a sequence of non-negative functions
such that  $\eta_{n}\,dx$ are probability measures on $\mathbb{R}^d$
weakly convergent to $\nu$.

On $\mathbb{R}^d\times[0, T]$ we consider the Cauchy problem
$$
\partial_{t}u_{n}=\partial_{x_{i}}\partial_{x_{j}}(a_{n}^{ij}u_{n})-\partial_{x_{i}}(b_{n}^{i}u_{n})
+c_{n}u_{n},\quad u_{n}|_{t=0}=\eta_{n}.
$$
Let us rewrite it as
\begin{equation*}
\partial_{t}u_{n}=a_{n}^{ij}\partial_{x_{i}x_{j}}u_{n}
+\left(2\partial_{x_{j}}a_{n}^{ij}-b_{n}^{i}\right)\partial_{x_{i}}u_{n}
+\left(\partial_{x_{i}x_{j}}a_{n}^{ij}-\partial_{x_{i}}b_{n}^{i}+q_{n}\right)u_{n},
\quad u_{n}|_{t=0}=\eta_{n}.
\end{equation*}

Here all the coefficients are smooth and bounded for each $n$.
It is well-known (see \cite{Fr}, Chapter 1,~\S7, Theorem 12) that there exists a
smooth bounded non-negative classical solution $u_n\in
L^{1}(\mathbb{R}^{d}\times[0,T])$. Further, a classical solution  $\{u_n\}$ is also
a weak solution in sense of distributions, thus  for every function
$\psi\in C_{0}^{\infty}(\mathbb{R}^{d})$ one has
\begin{equation}\label{eq-un}
\int_{\mathbb{R}^d}\psi(x)u_{n}(x,t)\,dx=
\int_{\mathbb{R}^d}\psi(x)\eta_{n}(x)\,dx +\int_{0}^{t}\int_{\mathbb{R}^d}
L_{n}(x, s)\psi u_{n}(x,s)\,dx\,ds.
\end{equation}
Let $\zeta\in
C_{0}^{\infty}\left(\mathbb{R}^d\right)$ be such that  $\zeta(x)=1$ if  $|x|\le
1$, $\zeta(x)=0$ if $|x|>2$, $|\zeta|\le 1$ and it has two bounded derivatives.
Substitute  $\psi(x)=\zeta(x/N)$ in the equality \ref{eq-un}
and let $N$ go to infinity. For every fixed $n$
the functions $a_{n}^{ij},\, b_{n}^{i},\, c_{n}$ are globally bounded, thus
Lebesgue's dominated theorem yields
\begin{equation*}
\int_{\mathbb{R}^d}u_{n}(x,t)\,dx =\int_{\mathbb{R}^d}\eta_{n}(x)\,
dx+\int_{0}^{t}\int_{\mathbb{R}^d}c_{n}(x,s)u_{n}(x,s)\, dx\, ds.
\end{equation*}
In particular, the measures $u_{n}(x,t)\, dx$ are subprobability
measures for all $t\in[0, T]$. Taking into account that  $\eta_n=0$ outside  $D$
and $c_n\le 0$, we obtain
\begin{equation}\label{con1ap}
\int_{D}u_{n}(x,t)\,dx\le \int_{D}\eta_{n}(x)\,
dx+\int_{0}^{t}\int_{D}c_{n}(x,s)u_{n}(x,s)\, dx\, ds.
\end{equation}

2. We choose a convergent subsequence in  $\{u_{n}\}$.

We observe that for fixed $k$ and sufficiently large $n$
the estimate $(A_n(x, t)y, y)\ge |y|^2\min\{m_{k+1}, 1\}$ holds for all  $(x, t)\in D_k\times[0, T]$
and $y\in\mathbb{R}^d$. Indeed,
$$ (a^{ij}I_n+(1-I_n)\delta^{ij})*\omega_{1/n}(x, t)=
(a^{ij}I_{k+1}+(1-I_{k+1})\delta^{ij})*\omega_{1/n}(x, t)
$$
if $(x, t)\in D_k\times[0, T]$ and
$n$ is large enough so that   $\mbox{supp}\,(y,
\tau)\mapsto\omega_{1/n}(y-x, \tau-t) \subset D_{k+1}\times(-1, T+1)$.
Similarly, for fixed  $n$ (large enough) and $k$  one has
$$
\|a^{ij}_n\|_{L^{\infty}(D_k\times[0, T])}\le
\|a\|_{L^{\infty}(D_{k+1}\times[0, T])}+1,\quad
\|b^i_n\|_{L^{\infty}(D_k\times[0, T])}\le
\|b^i\|_{L^{\infty}(D_{k+1}\times[0, T])}, $$ $$
\|c_n\|_{L^{\infty}(D_k\times[0, T])}\le
\|c\|_{L^{\infty}(D_{k+1}\times[0, T])}.
$$
Due to  \cite[Corollary  3.2]{BKR}, for every $k>2$ one has
$$
\int_{D_k\times [Tk^{-1}, T(1-k^{-1})]}u_{n}^{(d+1)/d}\, dx\, dt\le C_k,
$$
where  $C_k$ depends only on  $m_{k+1}$, $\|a\|_{L^{\infty}(D_{k+1}\times[0, T])}$,
$\|b^i\|_{L^{\infty}(D_{k+1}\times[0, T])}$, $\|c\|_{L^{\infty}(D_{k+1}\times[0, T])}$
and does not depend on $n$.

Since the unit ball in  $L^{(d+1)/d}$ is weakly compact, for every $k>2$ one can
extract from $\{u_{n}\}$ a  subsequence weakly convergent in
$L^{(d+1)/d}(D_k\times [Tk^{-1}, T(1-k^{-1})])$.
Without loss of generality, using the diagonal procedure,  we can assume that
$\{u_{n}\}$  converges weakly to a non-negative function  $u$ that belongs to
$L^{(d+1)/d}(D_k\times [Tk^{-1}, T(1-k^{-1})])$ for every $k$.

3. As we have shown above, for every $k$ the coefficients $a_{n}^{ij}$,
$b_{n}^{i}$ and $c_{n}$ are uniformly  bounded ( with respect to $n$) on
$D_k\times[0, T]$. Let  $\psi\in C^{\infty}_0(D)$.
Then $\mbox{supp}\,\psi \subset D_k$
for some  $k$ and there exists  a number  $C(\psi)$
(independent of $n$) such that
$$
\left|\int_{D}\psi(x)u_{n}(x,t)\,
dx- \int_{D}\psi(x)u_{n}(x,s)\, dx\right| \\
=\left|\int_{s}^{t}\int_{D}\left(L_{n}\psi(x, \tau)\right)u_{n}(x,
\tau)dx\,d\tau\right| \le C(\psi)|t-s|
$$
for all  $n$ and $s,
t\in[0, T]$. Hence the functions $$ f_{n}(t):=\int_{D}\psi(x)u_{n}(x,t)\, dx $$
are Lipschitzian with constant  $C(\psi)$ independent of $n$.
Thus, this is a uniformly bounded and equicontinuous family of functions
for every fixed
$\psi$. Hence,
by to Arzel\`a--Ascoli theorem there exists a subsequence, uniformly convergent on $[0, T]$. We observe that in the space
$L^{(d+1)/d}([0,T])$ the same subsequence converges to
$$
f(t):=\int_{D}\psi(x)u(x,t)\, dx $$ and, since the weak and uniform limits coincide
a.e., every subsequence of  $\{f_{n}\}$ (and thus the whole sequence as well) converges
uniformly to the same Lipschitzian function $\widetilde{f}$, which coincides with
$f$ on a full Lebesgue measure set in  $[0,T]$. Obviously this set depends on
$\psi$. Denote it by  $\mathbb{T}(\psi)$.

4. Let us show that $u$ constructed at Step  2 is a solution to (\ref{e1}).
Fix  $\psi\in C_{0}^{\infty}(\mathbb{R}^{d})$ with  support in some  $D_k$. We
have  $L\psi\in L^{\infty}(D_k\times[0, T])$ by the assumptions of the theorem.
Moreover, there is  a number  $C_k$ such that
$\sup_n\|L_n\psi\|_{L^{\infty}(D_k\times[0, T])}\le C_k$, and $L_n\psi$
converges to  $L\psi$ in $L^p(D_k\times[0, T])$ for every $p\ge 1$. Take  $t$
in the full measure set  $\mathbb{T}(\psi)$; for every element of this set  convergence
$$ \int_D\psi(x)u_{n}(x,t)\, dx\to\int_D\psi(x)u(x,t)\, dx $$
takes place. Let  $0<s<t$.
One has
\begin{multline*}
\Bigl|\int_D\psi(x)u_{n}(x,t)\, dx-
\int_D\psi(x)\eta_{n}(x)\, dx -\int_{s}^{t}\int_D L_{n}\psi u_{n}(x,\tau)\, dx\,
d\tau\Bigr|=
\\
=\Bigl|\int_D\psi(x)u_{n}(x,s)\, dx-\int_D\psi(x)\eta_{n}(x)\,dx\Bigr|\le C(\psi)s,
\end{multline*}
where  $C(\psi)$ is independent of $n$ and $s$.
Hence,
\begin{equation}\label{ner}
\Bigl|\int_D\psi(x)u_{n}(x,t)\, dx- \int_D\psi(x)\eta_{n}(x)\, dx -
\int_{s}^{t}\int_D L_{n}\psi(x, \tau)u_{n}(x,\tau)\, dx\, d\tau\Bigr|\le
C(\psi)s.
\end{equation}
We observe that
$$ \lim_{n\to\infty}\int_{s}^{t}\int_D L_{n}\psi(x, \tau)u_{n}(x,\tau)\, dx\, d\tau=
\int_{s}^{t}\int_D L\psi(x, \tau)u(x,\tau)\, dx\, d\tau.
$$
Indeed,
\begin{multline*}
\Bigl|\int_{s}^{t}\int_D (L_{n}\psi(x, \tau))u_{n}(x, \tau)\, dx\, d\tau-
\int_{s}^{t}\int_D (L\psi(x, \tau))u(x, \tau)\, dx\, d\tau\Bigr|\le \\
\le\|L_n\psi-L\psi\|_{L^{d+1}(D_k\times[s, t])}\|u_n\|_{L^{(d+1)/d}(D_k\times[s,
t])}+ \\ +\Bigl|\int_{s}^{t}\int_D (L\psi(x, \tau))u_{n}(x,\tau)\, dx\, d\tau-
\int_{s}^{t}\int_D (L\psi(x, \tau))u(x,\tau)\, dx\, d\tau\Bigr|,
\end{multline*}
where the first summand in the right-hand side tends to zero due to
convergence of $L_n\psi$ to $L\psi$ and the uniform norm boundedness of $\{u_n\}$,
shown above. The second summand tends to zero by the weak convergence of $\{u_n\}$
and the boundedness of $L\psi$. Thus, letting  $n\to\infty$ in (\ref{ner}),
we obtain
$$
\Bigl|\int_D\psi(x)u(x,t)\, dx- \int_D\psi(x)\,d\nu -
\int_{s}^{t}\int_D L\psi u \, dx\, d\tau\Bigr|\le C(\psi)s.
$$
Letting  $s$ to zero, we arrive at the equality
$$
\int_D\psi(x)u(x,t)\, dx=\int_D\psi(x)\,d\nu+\int_{0}^{t}\int_D L\psi u\, dx\, d\tau.
$$
Hence, the function  $u$ is a non-negative solution to the Cauchy problem (\ref{e1}).

5. Let us show that the measure  $u(x,t)dx\, dt$ is a solution in the class $\mathcal{M}_{\nu}$.
We recall that  $c_{n}\le0$ and $\eta_n\,dx$ are probability measures.
Due to  (\ref{con1ap}), for every function  $\psi\in C_{0}^{\infty}(D)$, $0\le\psi\le1$
the following inequality holds:
\begin{equation}\label{ner1}
\int_D\psi(x)u_{n}(x,t)\, dx-\int_{0}^{t}\int_D\psi(x)c_{n}(x,s)u_{n}(x,s)\, dx\, ds\le1.
\end{equation}
Let $\psi_{N}\in C^{\infty}_0(D)$, $0\le\psi_N\le 1$ and $\psi_N(x)=1$ if $x\in D_N$.
Let $t$ also belong to the full measure set
$\mathbb{T}=\bigcap\limits _{N\in\mathbb{N}}\mathbb{T}\left(\psi_{N}\right)$,
i.e.,  for all  $N\in\mathbb{N}$ the following convergence takes place:
$$
\int_D\psi_{N}(x)u_{n}(x,t)\, dx\to\int_D\psi_{N}(x)u(x,t)\, dx.
$$
Substituting such $\psi_{N}$ and $t$ into  (\ref{ner1}) and letting  $n\rightarrow\infty$,
we obtain
$$
\int_{D}\psi_{N}(x)u(x,t)\, dx-\int_{0}^{t}\int_{D}\psi_{N}(x)c(x,s)u(x,s)\, dx\, ds\le 1.
$$
Finally, letting  $N\rightarrow\infty$ and applying Fatou's lemma, we arrive at the required inequality
$$
\int_D u(x,t)\, dx-\int_{0}^{t}\int_D c(x,s)u(x,s)\, dx\, ds\le 1=\nu(D).
$$
This completes the proof.
\end{proof}

\begin{remark}\rm
Due to \cite[Corollary  3.2]{BKR}, under the assumptions of Theorem \ref{th1-1}
every solution  $\mu$ in $\mathcal{M}_{\nu}$ is given by a density
$\varrho\in L^{(d+1)/d}_{loc}(D\times(0, T))$
with respect to Lebesgue measure.
\end{remark}

The assumptions of local boundedness can be weakened, but under the additional assumption
that the elements of the matrix $A$ are Sobolev regular. The next  theorem
was proved in \cite{BDR} in the case $c=0$ under the assumption of
existence of a Lyapunov
function. We
give here a different and shorter proof and do not impose
global restrictions on the
coefficients.

\begin{theorem}\label{th1-2}
Let $p>d+2$. Assume that  for every $k$ we have
  $a^{ij}(\, \cdot \,, t)\in W_{loc}^{p,1}(D_k)$,
$$
\sup_{t\in(0, T)}\|a^{ij}(\,\cdot\,, t)\|_{W^{1, p}(D_k)}<\infty
$$
and $(A(x, t)y, y)\ge m_k|y|^2$ for all $(x, t)\in D_k\times[0, T]$,
$y\in\mathbb{R}^d$ and
some $m_k>0$. Assume also that
$b\in L^{p}(D_k\times [0, T])$ and $c\in L^{p/2}(D_k\times [0, T])$
for each number $k$.
Then  for every probability measure  $\nu$ the set $\mathcal{M}_{\nu}$ is not
empty.
\end{theorem}
\begin{proof}
1. Exactly as at Step 1 of the proof of the
previous theorem, we construct sequences of  smooth bounded functions
$a^{ij}_n$, $b^i_n$ and $c_n$ such that for every $D_k$
one has
$$
\lim_{n\to\infty}\|a^{ij}_n-
a^{ij}\|_{L^{p}(D_k\times [0, T])}=0,
\quad \lim_{n\to\infty}\|b^i_n-
b^i\|_{L^{p}(D_k\times [0, T])}=0,
\quad \lim_{n\to\infty}\|c_n-
c\|_{L^{p/2}(D_k\times [0, T])}=0,
$$
in particular the norms
$\|a^{ij}_n\|_{L^p(D_k\times [0, T])}$, $\|b^i\|_{L^p(D_k\times [0, T])}$,
$\|c_n\|_{L^{p/2}(D_k\times [0, T])}$ are bounded uniformly in $n$.
Moreover, $c_n\le 0$ and $(A_n(x, t)y, y)\ge \min\{m_{k}, 1\}$ for all  $n>k$.
We extend  the measure $\nu$ by zero
outside  $D$ to a measure on $\mathbb{R}^d$. Let  $\eta_{n}\in C_{0}^{\infty}(D)$
be a sequence of
non-negative functions such that
$\eta_{n}\,dx$ are probability measures on $\mathbb{R}^d$
weakly convergent to $\nu$.

Let  $\{u_n\}$ be a smooth
bounded solution of the  Cauchy problem
$$
\partial_{t}u_{n}=\partial_{x_{i}}\partial_{x_{j}}(a_{n}^{ij}u_{n})-
\partial_{x_{i}}(b_{n}^{i}u_{n}) +c_{n}u_{n},\quad u_{n}|_{t=0}=\eta_{n},
$$
and $u_n(x, t)\,dx$ are subprobability measures for every $t$ and
$$
\int_{\mathbb{R}^d}u_n(x, t)\,dx=\int_{\mathbb{R}^d}\eta_n(x)\,dx
+\int_0^t\int_{\mathbb{R}^d}c_n(x, t)u_{n}(x, t)\,dx\,d\tau.
$$

2. We choose a convergent subsequence in  $\{u_{n}\}$.
Due to  \cite[Corollary3.9]{BKR}, for every $k>2$ the following estimate
on the H\"older norm holds:
$$
\|u_{n}\|_{C^{\alpha}(D_k\times [Tk^{-1}, T(1-k^{-1})])}\le C_k,
$$
where
$\alpha\in(0, 1)$ and $C_k$ are independent of  $n$. The  Arzel\`a--Ascoli
theorem, the diagonal method and a passage to a subsequence enable us to conclude that
the sequence $\{u_{n}\}$ converges uniformly to some function $u$ on
$D_k\times[Tk^{-1}, T(1-k^{-1})]$ for every $k$. It is obvious that  $u$ is a
non-negative continuous function. Let us show that  $u$ satisfies (\ref{e1}).

Let  $\psi\in C_{0}^{\infty}(D)$. Then  $\mbox{supp}\,\psi\subset D_k$
for some $k$. The uniform convergence of  $\{u_{n}\}$ immediately yields that
$$
\lim_{n\to\infty}\int_D\psi(x)u_{n}(x,t)dx=\int_D\psi(x)u(x,t)dx
$$
for all $t\in(0, T)$. Now let $0<s<t<T$.
Observe that
\begin{multline*}
\Bigl|\int_s^t\int_{D}L_n\psi u_n\,dx\,d\tau-\int_s^t\int_{D}L\psi
u\,dx\,d\tau\Bigr|\le \\ \le\|L_n\psi-L\psi\|_{L^1(D_k\times[s,
t])}\|u_n\|_{L^{\infty}(D_k\times[s, t])} +\|L\psi\|_{L^1(D_k\times[s, t])}\|u-
u_n\|_{L^{\infty}(D_k\times[s, t])},
\end{multline*}
where the first summand in
right-hand side tends to zero by convergence of $a^{ij}_n$, $b^i_n$ and
$c_n$ to  $a^{ij}$, $b^i$ and $c$, respectively. The uniform convergence of $\{u_n\}$
to  $u$ yields convergence of the second summand to zero. Hence,
$$
\lim_{n\to\infty}\int_s^t\int_{D}L_n\psi u_n\,dx\,d\tau=\int_s^t\int_{D}L\psi
u\,dx\,d\tau.
$$
Thus, letting  $n\to\infty$, we obtain
$$
\int_D\psi(x)u(x, t)=\int_D\psi(x)u(x, s)\,dx+\int_s^t\int_DL\psi u\,dx\,d\tau
$$
for all  $s, t\in (0, T)$.

3. Let us justify the limit as  $s\to 0$.

Let  $0<\tau<T$ and $y\in C_{0}^{\infty}(D)$. We extend the function  $y$
by zero outside $D$. Let also $w_{n,\tau}$ be a solution of the adjoint
problem
$$
\partial_{t}w_{n,\tau}+a_{n}^{ij}\partial_{x_{i}}\partial_{x_{j}}w_{n,\tau}
+b_{n}^{i}\partial_{x_{i}}w_{n,\tau}+c_{n}w_{n,\tau}=0,\quad
w_{n,\tau}|_{t=\tau}=y.
$$
Let $\zeta\in C_{0}^{\infty}\left(\mathbb{R}\right)$
be such that  $\zeta(x)=1$ if $|x|\le1$, $\zeta(x)=0$ if $|x|>2$, $|\zeta|\leq1$
and let $\zeta$ have two bounded derivatives. Multiplying the adjoint equation
by $\zeta_Nu_n$ and integrating by parts and then letting $N$ to infinity, we
obtain
$$
\int_{D} y(x)u_{n}(x,\tau)\, dx=\int_{D}w_{n,\tau}(x,0)\eta_{n}(x)\,dx,
$$
in the latter equality we have taken into account that
$\mbox{supp}\,y,\quad \mbox{supp}\,\eta_n  \subset D$.
By \cite[Part III, Theorem 10.1]{Ladyzh}, for
every ball  $U\subset\mathbb{R}^d$ the estimate
$$
\|w_{n,\tau}(x,0)-y(x)\|_{L^{\infty}(U)}\le C(U)\tau^{\alpha}
$$
holds with  $C$ and $\alpha$ independent of $n$.
Thus,
\begin{multline*}
\Bigl|\int_{D}y(x)u_n(x, \tau)\,dx-
\int_{D}y\,d\nu\Bigr|\le \\
\le\int_{D}|w_{n, \tau}(x, 0)-y(x)|\eta_n(x)\,dx
+\Bigl|\int_{D}y(x)\eta_n(x)\,dx- \int_{D}y\,d\nu\Bigl|\le \\ \le
C\tau^{\alpha}+\Bigl|\int_{D}y(x)\eta_n(x)\,dx- \int_{D}y\,d\nu\Bigl|.
\end{multline*}
Letting  $n\to\infty$, we obtain the estimate
$$
\Bigl|\int_{D}y(x)u(x, \tau)\,dx-\int_{D}y\,d\nu\Bigr|\le C\tau^{\alpha},
$$
which yields that
$$
\lim_{\tau\to 0}\int_{D}y(x)u(x, \tau)\,dx=\int_{D}y\,d\nu.
$$
Finally, the inequality
$$
\int_{D} u(x,t)\, dx-
\int_{0}^{t}\int_{D} c(x,s)u(x,s)\, dx\, ds\le1=\nu(D)
$$
can be justified exactly in the same way at Step 5 of the proof of Theorem \ref{th1-1}.
\end{proof}

\begin{remark}\rm
By \cite[Corollary  3.9]{BKR}, under the assumptions of Theorem \ref{th1-2}
every solution  $\mu$ in $\mathcal{M}_{\nu}$ is given by a locally H\"older continuous density  $\varrho$
with respect to Lebesgue measure.
\end{remark}

Now we proceed to the case of a degenerate matrix $A$.

\begin{theorem} \label{th1-3}
Let the coefficients  $a^{ij},\, b^{i}$ and $c$ be continuous in $x$, measurable in $t$
and bounded on  $D_k\times[0, T]$ for every $k$.
Assume that the diffusion matrix~$A$ is symmetric and
$(A(x, t)y, y)~\ge 0$ for all $(x, t)\in D\times[0, T]$ and $y\in\mathbb{R}^d$.
Then, for every probability measure $\nu$, the set  $\mathcal{M}_{\nu}$ is not empty.
\end{theorem}
\begin{proof}
We use the well-known method of  "vanishing viscosity".

1. Introduce the operator
$
L_{\varepsilon}:=\varepsilon\Delta+L
$
for every $\varepsilon>0$ and consider the Cauchy problem
\begin{equation}\label{eq:visc-1}
\partial_{t}\mu_{t}=L_{\varepsilon}^{*}\mu, \quad \mu|_{t=0}=\nu.
\end{equation}
Obviously, the hypotheses of Theorem \ref{th1-1} are fulfilled.
Then, for every $n$, the problem (\ref{eq:visc-1})
with $\varepsilon=1/n$ has a solution  $\mu^n$ given by a flow of
 subprobability measures $(\mu_{t}^{n})_{t\in(0,T)}$ on $D$
for which the following  inequality holds:
$$
\mu_t^n(D)\le \nu(D)+\int_0^t\int_{D}c(x, s)\mu_s^n(dx)\,ds.
$$

2. Choosing a convergent subsequence of solutions.

There exists a subsequence of indices $n_{l}$ such that the measures
 $\mu_{t}^{n_{l}}$
converge weakly  on each compact set  $\overline{D_k}$  for each  $t\in[0,T]$.
To prove this, it suffices to apply Prokhorov's theorem for any fixed compact set for a dense set
in $S\in[0,T]$ %#?KTO S?
and the diagonal method, and then, using again the diagonal
method for compact sets $D_k$, we can extract a subsequence of
measures that converges on all compact sets.
Let us show that the constructed subsequence is a Cauchy sequence for every  $t\in[0,T]$. Let
$t\in[0,T]$, $s\in S$ and $\varphi\in C_{0}^{\infty}(\mathbb{R}^{d})$.
By the boundedness of the coefficients of $L$ on cylinders, one has
\begin{multline*}
\left|\int_D\varphi d\mu_{t}^{n_{p}}-\int_D\varphi d\mu_{t}^{n_{k}}\right|\le
\left|\int_D\varphi d\mu_{t}^{n_{p}}-\int_D\varphi d\mu_{s}^{n_{p}}\right|+\\
+\left|\int_D\varphi d\mu_{s}^{n_{p}}-\int_D\varphi d\mu_{s}^{n_{k}}\right|+
\left|\int_D\varphi d\mu_{s}^{n_{k}}-\int_D\varphi d\mu_{t}^{n_{k}}\right|\le\\
\le 2C(\varphi)\cdot|t-s|+\left|\int_D\varphi d\mu_{s}^{n_{p}}-\int_D\varphi d\mu_{s}^{n_{k}}\right|.
\end{multline*}
Given $\varepsilon>0$, we can choose $s$ close enough to $t$ to make the first summand
 less than $\varepsilon/2$.
Since the sequence $\mu_{s}^{n_{l}}$ converges, it is a Cauchy sequence,
thus there exists  a number  $N$ such that for all
 $p,k>N$ the second summand is less than  $\varepsilon/2$. Therefore, it is proved
 that for every function $\varphi\in C_{0}^{\infty}(D)$
the sequence of integrals
${\displaystyle \int_D\varphi\, d\mu_{t}^{n_{l}}}$ is a Cauchy sequence.
Fix a number $k$. Since every  function $f$ continuous on a compact set  $\overline{D_k}$
can be uniformly approximated by functions in $C^{\infty}_0(D)$ on $\overline{D_k}$,
the sequence of integrals
${\displaystyle \int_{D_k} f\, d\mu_{t}^{n_{l}}}$
is a Cauchy sequence. Hence for every $t$
the sequence $\mu_{t}^{n_l}$ converges weakly to some subprobability measure $\mu_t$
on each $D_k$.
We observe that for every continuous function
$f$ the mapping ${\displaystyle t\mapsto\int_{D_k} f\, d\mu_{t}}$ is Borel
 measurable on $[0,T]$ as a limit of measurable functions. Consider the class
$\Phi$ of bounded Borel functions $\varphi$ on $D_k$,
for which the mapping  ${\displaystyle t\mapsto\int_{D_k}\varphi\, d\mu_{t}}$
is Borel measurable on $[0,\tau]$. The set $\Phi$ contains the algebra of continuous bounded
functions on  $D_k$ and is closed with respect to uniform and monotone limits.
By the monotone class theorem (see \cite[Theorem 2.2.12]{BS}) the set  $\Phi$
contains all bounded Borel functions on~$D_k$. In particular, the mapping
$t\mapsto\mu_{t}(B)$ is Borel measurable on $[0,T]$ for each Borel set
 $B\subset D_k$. Since $k$ was arbitrary, this is true for every
Borel subset of $D$. Let $\mu$ be the measure given by a family $(\mu_{t})_{t\in[0,T]}$.
Obviously,  $\mu^{n_{l}}$ converges weakly  to $\mu$ on  $D_k\times[0, T]$ for every $k$.

3. Passing to the subsequence constructed above, we conclude that
 the sequence  $\mu^n=(\mu_t^n)_{t\in(0, t)}$ is such that
$\mu^n$  converges weakly to  $\mu=(\mu_t)_{t\in(0, T)}$ and
$\mu_t^n$  converge weakly to $\mu_t$ for each $t$.

Let $\varphi\in C_{0}^{\infty}(D)$.
For every $n$ one has
\begin{equation}\label{eq:vjaz-1}
\int_{D}\varphi d\mu_{t}^{n}-
\int_{D}\varphi d\mu_{0}=
\int_{0}^{t}\int_{D}L_{1/n}\varphi d\mu_{s}^{n}ds.
\end{equation}
Since
$$
\Bigl|\int_{0}^{t}\int_{D}\left(L_{1/n}-L\right)\varphi d\mu_{s}^{n}ds\Bigr|\le
\frac{T}{n}\sup_{D}\left|\Delta\varphi\right|,
$$
by  the weak convergence of $\mu^{n}$ to $\mu$ we conclude that
$$
\lim\limits _{n\rightarrow\infty}\int_{0}^{t}\int_{D}L_{1/n}\varphi d\mu_{s}^{n}ds=\int_{0}^{t}\int_{D}L\varphi d\mu_{s}ds,
$$
and also for every  $t$
$$
\lim\limits _{n\rightarrow\infty}\int_{D}\varphi d\mu_{t}^{n}=\int_{D}\varphi d\mu_{t}.
$$
Hence one can let  $n\to\infty$ in (\ref{eq:vjaz-1}) and obtain
$$
\int_D\varphi d\mu_{t}-\int_D\varphi d\mu_{0}=\int_{0}^{t}\int_D L\varphi d\mu_{s}ds.
$$
This means that  $\mu$ is the required solution to the Cauchy problem (\ref{e1}).
Finally, the inequality
$$
\int_{D} u(x,t)\, dx-\int_{0}^{t}\int_{D} c(x,s)u(x,s)\, dx\, ds\le1=\nu(D)
$$
is justified as at Step 5 of the proof of  Theorem \ref{th1-1}.
\end{proof}

\begin{remark}\label{rm1}\rm
If we assume that the coefficients and the initial data are more regular,
then one can construct a solution
given by a density even in the case of a  degenerate diffusion matrix.
Such results were obtained for $D=\mathbb{R}^d$ in  \cite{DL, Figalli, LeBrL}.
These results can be extended to the case of arbitrary domains $D$ as follows.
Let  $q\ge 1$. Suppose that for every $k$ the following assumptions are fulfilled:
$b, c\in L^{q}(D_k\times[0, T])$,
$a^{ij}(\, \cdot\,,t)\in W^{1, q}(D_k)$ and
$$
\sup_{t\in[0, T]}\|a^{ij}\|_{W^{1, q}}(D_k)<\infty.
$$
Let $p=(q-1)/q$. Suppose that $(pc+(p-1){\rm div} h)^{+}\in L^1([0, T], L^{\infty}(D))$, where
$
h^i=\partial_{x_j}a^{ij}-b^i.
$
Let also $\varrho_0\in L^p(D)$ and $\nu=\varrho_0(x)\,dx$.
Then, there exists a solution $\mu\in \mathcal{M}_{\nu}$
given by a density $\varrho\in L^{\infty}([0, T], L^p(D))$.

The proof repeats practically verbatim the reasoning from  \cite{DL} and \cite{LeBrL}.
We assume that the coefficients are locally smooth (for that if suffices to convolute
the equation with a smooth kernel). Thus, the only problem is the
 unboundedness of the coefficients on  $D$ and the degenerate diffusion matrix.
  Let $\psi_k\in C^{\infty}_0(D)$ and let $\psi_k(x)=1$ if $x\in D_k$.
Let $p>1$. Consider the Cauchy problem $\partial_tu_k=L_k^{*}u_k$, $u_k|_{t=0}=u_0$, where
$$
L_k=(\psi_ka^{ij}+k^{-1}\delta^{ij})\partial_{x_i}\partial_{x_j}+\psi_kb^i\partial_{x_i}+
p^{-1}(p-1)\bigl((2\partial_{x_j}a^{ij}-b^i)\partial_{x_i}\psi_k+a^{ij}\partial_{x_i}\partial_{x_j}\psi_k\bigr)
+\psi_kc.
$$
Observe that there exists  a solution $\{u_k\}$ of this Cauchy problem which is a smooth function
and $u_k\in L^{\infty}([0, T], L^1(D))$.
Then the inequality
$$
\partial_t |u_k|^p\le
\partial_{x_i}\partial_{x_j}((\psi_ka^{ij}+k^{-1}\delta^{ij})|u_k|^p)-
\partial_{x_i}(\psi_kb|u_k|^p)+\psi_k(pc+(p-1){\rm div} h)|u_k|^p
$$
and the Grownwall inequality immediately yield that
$$
\sup_{t\in[0, T]}\int_{D}|u_k(x, t)|^p\,dx\le M
$$
with a constant $M$ independent of $k$. Further in the standard way
one can extract a subsequence from $u_k$ which converges to a solution $\varrho$
 of class $L^{\infty}([0, T], L^p(D))$.
\end{remark}

To summarize this section, we find out when the \emph{identity}
\begin{equation}\label{main-e}
\mu_t(D)=\nu(D)+\int_0^t\int_{\mathbb{R}^d}c(x, s)\,d\mu_s\,ds
\end{equation}
holds instead of an inequality (in particular, this means that if  $c=0$ and $\nu$ is a probability measure,
 then $\mu_t$ are also probability measures).

\begin{theorem}\label{th-est}
Let $\mu=(\mu_t)_{0<t<T}\in \mathcal{M}_{\nu}$ and $c\le 0$.
Suppose that there exists a function  $V$ such that
$V\in C^{2, 1}(D\times(0, T))\bigcap C(D\times[0, T))$, for every interval
 $[\alpha, \beta]\in(0, T)$
$$
\lim_{k\to\infty}\inf_{D_k\setminus D_{k-1}\times[\alpha, \beta]}V(x, t)=+\infty
$$
and for some functions $K, H\in L^1((0, T))$ with $H\ge 0$, the following estimate
holds:
$$
\partial_tV(x, t)+LV(x, t)\le K(t)+H(t)V(x, t).
$$
Suppose also that  $V(\, \cdot \,,0)\in L^1(\nu)$.
Then for a.e.  $t\in(0, T)$
$$
\mu_t(D)=\nu(D)+\int_0^t\int_{D}c(x, s)\,d\mu_s\,ds
$$
and the estimate
$$
\int_{D}V(x, t)\,d\mu_t\le Q(t)+R(t)\int_{D}V(x, 0)\,d\nu
$$
holds, where
$$
R(t)=\exp\Bigl(\int_0^tH(s)\,ds\Bigr), \quad Q(t)=R(t)\int_0^t\frac{K(s)}{R(s)}\,ds.
$$
\end{theorem}
\begin{proof}
Let $\zeta_N\in C^2([0, +\infty))$ be such that  $0\le \zeta'\le 1$, $\zeta''\le 0$,
where $\zeta_N(s)=s$ if $s\le N-1$ and $\zeta(s)=N$ if $s>N+1$.
Let also $\eta\in C^{\infty}_0((0, T))$.
For the function $u(x, t)=(\zeta_{N}(V(x, t))-N)\eta(t)$ and the solution
 $\mu=(\mu_t)_{0<t<T}\in\mathcal{M}_{\nu}$
we have
$$
\int_0^T\int_D[\partial_tu+Lu]\,d\mu_t\,dt=0,
$$
which yields that
$$
-\int_0^T\eta'(t)\int_{D}(\zeta_{N}(V(x, t))-N)\,d\mu_t\,dt=\int_0^T\eta(t)\int_{D}L(\zeta_{N}(V(x, t))-N)\,d\mu_t\,dt.
$$
Since $\eta$ is arbitrary,
$$
\frac{d}{dt}\int_D(\zeta_{N}(V(x, t))-N)\,d\mu_t=\int_{D}L(\zeta_{N}(V(x, t))-N)\,d\mu_t.
$$
Hence there holds the identity
\begin{multline*}
\int_{D}\zeta_{N}(V(x, t))\,d\mu_t=\int_{D}\zeta_{N}(V(x, s))\,d\mu_s+
\\
\left(\mu_t(D)-\nu(D)-\int_0^t\int_{D}c(x, \tau)\,d\mu_{\tau}\,d\tau\right)N+
\\
+\int_s^t\int_{D}\Bigl(\zeta_N'(V)(\partial_tV+LV)+\zeta_N''(V)|\sqrt{A}\nabla V|^2\Bigr)\,d\mu_{\tau}\,d\tau+
\\
+\int_s^t\int_{D}c\left(\zeta_N(V)-\zeta'_N(V)V\right)\,d\mu_{\tau}\,d\tau.
\end{multline*}
Observing that $z\zeta'_N(z)\le \zeta_N(z)$, we arrive at
\begin{multline*}
\int_{D}\zeta_N(V(x, t))\,d\mu_t\le \int_{D}\zeta_N(V(x, s))\,d\mu_s+
\\
\left(\mu_t(D)-\nu(D)-\int_0^t\int_{D}c(x, \tau)\,d\mu_{\tau}\,d\tau\right)N+
\\
\int_s^t K(\tau)+H(\tau)\int_{D}\zeta_N(V(x, \tau))\,d\mu_{\tau}\,d\tau,
\end{multline*}
Letting $s\to 0$, we obtain
\begin{multline}\label{in1}
\int_{D}\zeta_N(V(x, t))\,d\mu_t\le \int_{D}\zeta_N(V(x, 0))\,d\nu+
\\
\left(\mu_t(D)-\nu(D)-\int_0^t\int_{D}c(x, \tau)\,d\mu_{\tau}\,d\tau\right)N+
\\
\int_0^tK(\tau)+H(\tau)\int_{D}\zeta_N(V(x, \tau))\,d\mu_{\tau}\,d\tau.
\end{multline}
Since
$$
\mu_t(D)\le \nu(D)+\int_0^t\int_{D}c(x, s)\,d\mu_s\,ds,
$$
the last inequality can be rewritten in the following way:
$$
\int_{D}\zeta_N(V(x, t))\,d\mu_t\le \int_{D}\zeta_N(V(x, 0))\,d\nu+
\int_0^tK(\tau)+H(\tau)\int_{D}\zeta_N(V(x, \tau))\,d\mu_{\tau}\,d\tau.
$$
Using Grownwall's inequality, we have
$$
\int_{D}\zeta_N(V(x, t))\,d\mu_t\le Q(t)+R(t)\int_{D}\zeta_N(V(x, 0))\,d\nu.
$$
Letting $N\to\infty$, we obtain the required estimate.

Moreover, if
$\mu_t(D)<\nu(D)+\displaystyle\int_0^t\int_{D}c(x, s)\,d\mu_s\,ds$,
then, letting  $N\to\infty$ in (\ref{in1}),
we obtain
$$
\int_{D}V(x, t)\,d\mu_t-\int_{D}V(x, 0)\,d\nu-
\int_0^tK(\tau)+H(\tau)\int_{D}V(x, \tau)\,d\mu_{\tau}\,d\tau=-\infty,
$$
which is impossible. Hence
$\mu_t(D)=\nu(D)+\displaystyle\int_0^t\int_{D}c(x, s)\,d\mu_s\,ds$.
The theorem is proved.
\end{proof}

Thus, in the case $c=0$ to construct a probability solution it suffices to
construct a subprobability solution (for that only a local regularity
of the coefficients is needed). This solution is automatically a probability solution if there is
a  Lyapunov function.

Observe that the assumption $V(\,\cdot\,,0)\in L^{1}(\nu)$ is not restrictive:
if there is \emph{some} Lyapunov function, then there is a Lyapunov function
 integrable with respect to the initial condition.
More precisely, the following generalization of a lemma from \cite{BDR} is true.

\begin{proposition}\label{lem2-Lyap_f}
Let $\mu=\left(\mu_{t}\right)_{t\in[0,T]}$
be a solution to the  Cauchy problem $\partial_{t}\mu=L^{*}\mu$ and $\mu|_{t=0}=\nu$,
where $\nu$ is a probability measure on $D$. Suppose that there exists  a non-negative
function $V$ such that
$V\in C^{2, 1}(D\times(0, T))\bigcap C(D\times[0, T))$,
for every interval  $[\alpha, \beta]\in(0, T)$
$$
\lim_{k\to\infty}\inf_{D_k\setminus D_{k-1}\times[\alpha, \beta]}V(x, t)=+\infty
$$
and for some functions $K, H\in L^1((0, T))$ with $H\ge 0$, one has
$$
\partial_tV(x, t)+LV(x, t)\le K(t)+H(t)V(x, t).
$$
Then there exists a non-negative function $W\in C^{2, 1}(D\times(0, T))\bigcap C(D\times[0, T))$
such that for every interval $[\alpha, \beta]\in(0, T)$
$$
\lim_{k\to\infty}\inf_{D_k\setminus D_{k-1}\times[\alpha, \beta]}W(x, t)=+\infty,
$$
the following inequality holds:
$\partial_tW(x, t)+LW(x, t)\le K(t)+H(t)$ holds and there is an inclusion $W(x, 0)\in L^1(\nu)$.
\end{proposition}
\begin{proof}
Construct a non-negative function
$\theta\in C^{2}\left(\mathbb{R}\right)$ such that  $\theta(0)=0$,
$\lim\limits _{r\rightarrow\infty}\theta(r)=+\infty$, $0\leq\theta^{'}(r)\leq1$,
$\theta^{''}(r)\leq0$ and $\theta(V(\, \cdot\,, 0))\in L^{1}\left(\nu\right)$.
For that it suffices to find a function  $\theta$ with the properties listed above and
integrable with respect to the measure $\sigma=\nu\circ W^{-1}(\,\cdot\,, 0)$.
Find an increasing sequence of numbers $z_{k}$ such that
$z_{k+1}-z_{k}\geq z_{k}-z_{k-1}\geq1$
and $\sigma\left([z_{k},\infty)\right)\leq2^{-k}$.
Let $\theta_0$ be a linear function on each interval $[z_{k}, z_{k+1}]$
 with $\Theta_0\left(z_{k}\right)=k-1$.
We obtain a $\sigma-$integrable increasing concave function $\theta_{0}$.
However, it does not belong to the class  $C^{2}$. Take for  $\theta$ the function
$$
\theta(z)=\int_{0}^{z}g(s)ds,\quad g\in C^{1}(\mathbb{R}),
$$
where $g^{'}(z)\leq0$ and  $g(z)=\theta_{0}^{'}(z)$ if $z\in(z_{k}, z_{k+1}-k^{-1})$.
Obviously, it is the required function. Further, taking into account that $\theta$ is
concave and that $c$ is non-positive, we obtain
$$
\partial_{t}\theta(V)+L\theta(V)
=\theta^{'}(V)\bigl(\partial_{t}V+LV\bigr)
+\theta^{''}(V)\left(A\nabla V,\nabla V\right)+c\bigl(\theta(V)-\theta'(V)V\bigr)
\le K+H\theta(V).
$$
The function $W:=\log(1+\theta(V))$ is the required function.
\end{proof}

\begin{remark}\label{m-e}\rm
Modifying a bit the reasoning above, one can
get sufficient conditions for  (\ref{main-e}) of another type.
Let the function  $V$ be such that
$V\in C^{2}(D)$ and
$$
\lim_{k\to\infty}\inf_{D_k\setminus D_{k-1}}V(x)=+\infty.
$$
Let $\zeta\in C^{\infty}([0, +\infty))$, $\zeta(x)=1$ if $x<1$ and $\zeta(x)=0$ if $x>2$.

Set $\varphi_N(x)=\zeta(V(x)/N)$. If $\mu\in\mathcal{M}_{\nu}$, then
$$
\int_{D}\varphi_N\,d\mu_t=\int_{D}\varphi_N\,d\nu+\int_0^t\int_{D}[L_0\varphi_N+c\varphi_N]\,d\mu_s\,ds,
$$
where $L_0=a^{ij}\partial_{x_i}\partial_{x_j}+b^i\partial_{x_i}$.
Observe that $\varphi_N\to 1$ if $N\to\infty$. Thus,   (\ref{main-e}) is ensured by
the identity
$$
\lim_{N\to\infty}\int_0^t\int_{D}L_0\varphi_N\,d\mu_s\,ds=0.
$$
 It is fulfilled, for example, if
$$
\lim_{N\to\infty}\int_0^T\int_{N\le V\le 2N}N^{-1}|L_0 V|+N^{-2}|\sqrt{A}\nabla V|^2\,d\mu=0.
$$
If  $a^{ij}(\,\cdot\,t)\in W^{1, 1}_{loc}(D)$,
$\mu_t=\varrho(x, t)\,dx$ and $\varrho(\,\cdot \,, t)\in W^{1, 1}_{loc}(D)$, then
 (\ref{main-e}) is ensured, for example, by
$$
\lim_{N\to\infty}\int_0^T\int_{N\le V\le 2N}N^{-1}|(b-\beta_{\mu})\nabla V|+N^{-2}|\sqrt{A}\nabla V|^2\,d\mu=0,
$$
where
$\beta^{i}_{\mu}=\sum_{j=1}^{d}\bigl(\partial_{x_j}a^{ij}
+a^{ij}\varrho^{-1}\partial_{x_j}\varrho\bigr)$.

We also note that (\ref{main-e}) has a clear  probabilistic sense:
the diffusion process corresponding to the operator $L$ does not reach the
boundary of the domain $D$ in finite time.
\end{remark}

\section{\sc Uniqueness results}

In this section, we give some conditions sufficient for the uniqueness
and the resulting theorem on the  existence and uniqueness for the Cauchy problem (\ref{e1}).
We generalize the results of  \cite{BRSH-Servy, SHTVP11, SHDAN11}.
Note that some examples  non-uniqueness are given in \cite{BRSH-Servy}.

Since a proof  of uniqueness is usually divided into the two steps: a local
estimate and its application in some limiting procedure, we do not repeat this first step in details
(which can be found in the listed papers).

Let the matrix  $A(x, t)=(a^{ij}(x, t))_{1\le i, j\le d}$ be symmetric and satisfy the following condition:

\, (H1) \, for every $D_k\subset D$ there exist strictly positive numbers $m_k$ and $M_k$, such that
the estimate
$$
m_k|y|^2\le (A(x, t)y, y)\le M_k|y|^2
$$
holds for all $y\in\mathbb{R}^d$ and all  $(x, t)\in D\times(0, T)$ .

Let us recall the definition of the functional class VMO.

Let  $g$ be a bounded function on $\mathbb{R}^{d+1}$. Set
$$
O(g, R)=\sup_{(x, t)\in\mathbb{R}^{d+1}}\sup_{r\le R}r^{-2}|U(x, r)|^{-2}
\int_t^{t+r^2}\int\int_{y, z\in U(x, r)}|g(y, s)-g(z, s)|\,dy\,dz\,ds.
$$
If $\lim\limits_{R\to 0}O(g, R)=0$, then the function $g$ is said to belong to the class $VMO_x(\mathbb{R}^{d+1})$.

If $g\in VMO_x(\mathbb{R}^{d+1})$, then one can always assume that
$O(g, R)\le w(R)$ for all $R>0$, where $w$ is a continuous function
on $[0, +\infty)$ and $w(0)=0$.

Suppose that a function $g$ is defined on $D\times[0, T]$
and is bounded on $D_k\times[0, T]$ for every $k$. Let us extend $g$ by zero to all of $\mathbb{R}^{d+1}$.
If for every function $\zeta\in C^{\infty}_0(D)$ the function $g\zeta$
belongs to the class $VMO_x(\mathbb{R}^{d+1})$,
then we say that $g$ belongs to the class $VMO_{x, loc}(\mathbb{R}^d\times[0, T])$.

Set
$$
L_0u=a^{ij}\partial_{x_i}\partial_{x_j}u+b^i\partial_{x_i}u.
$$

\begin{theorem}\label{th-uniq-1}
Suppose that  $a^{ij}\in VMO_{x, loc}(D\times[0, T])$
and that the matrix
$A=(a^{ij})$ satisfies condition  {\rm (H1)}.
Suppose that there is a function $V=V(x)$ such that $V\in C^{2}(D)$ and
$$
\lim_{k\to\infty}\inf_{D_k\setminus D_{k-1}}V(x)=+\infty.
$$
Then the set
$\mathbb{M}_{\nu}$ consisting of the measures $\mu\in\mathcal{M}_{\nu}$
for which the functions $|L_0V|$, $|\sqrt{A}\nabla V|^2$ belong to
$L^1(\mu, D_k\times[0, T])$ for every $k$ and
$$
\lim_{N\to\infty}\int_0^T\int_{N\le V\le 2N}
N^{-1}|L_0 V|+N^{-2}|\sqrt{A}\nabla V|^2\,d\mu=0,
$$
contains at most one element.
\end{theorem}
\begin{proof}
The proof is the same as the proof of Theorem 3.1 from \cite{BRSH-Servy},
but there are some new aspects concerning the domain $D$ and the term $c$.
More precisely, local part of the proof of Theorem 3.1 is saved but
global part of the proof requires a new consideration.

We start with the case of the coefficients of
class $\bigcap\limits _{k\in\mathbb{N}}C^{\infty}(D_k\times[0, T])$.

1. Let  $\varphi_N(x)=\eta(V(x)/N)$, where  a nonnegative function  $\eta\in
C^{\infty}_0([0, +\infty))$  is such that $\eta(z)=1$ for $0\le z\le 1$ and
$\eta(z)=0$ for $z>2$. Moreover,  $0\le\eta\le 1$ and there is a number $K>0$
such that the estimate $|\eta'(z)|^2\eta^{-1}(z)\le K$ holds  for all  $x\in \mbox{supp}\,\eta$.
We observe that  $\varphi_N\in C^{\infty}_0(D)$,
in particular, there exists a positive integer $k_0=k_0(N)$ such that
$\mbox{supp}\,\varphi_N \subset D_{k_0}$.

We redefine the functions $a^{ij}$, $b^i$, $c$ outside $D_{k_0}\times[0, T]$ so
that the new functions are bounded together with all derivatives on
$\mathbb{R}^d\times[0, T]$.

2. Let $\psi\in C^{\infty}_0(D)$, $|\psi|\le 1$ and let $f$ be a solution of the adjoint Cauchy problem
$$
\partial_sf+Lf=0, \quad f|_{s=t}=\psi.
$$
Suppose that $\sigma_1$ and $\sigma_2$ are two solutions in the set $\mathbb{M}_{\nu}$.
Set $\mu=\sigma_1-\sigma_2$.
Multiplying the equation $\partial_t\mu=L^{*}\mu$ by $f\varphi_N$ and integrating, we arrive at the equality
$$
\int_{D}\psi\varphi_N\,d\mu_t=\int_0^t\int_{D}2(A\nabla\varphi_N, \nabla f)+fL_0\varphi_N\,d\mu_s\,ds.
$$

3. Set $\sigma=(\sigma_1+\sigma_2)/2$.
Let us estimate
$$
\int_{0}^t\int_{D}\varphi_N|\sqrt{A}\nabla f|^2\,d\sigma_s\,ds.
$$
Multiplying
the equation $\partial_t\sigma=L^{*}\sigma$ by  $f^2\varphi_N$
and integrating, we arrive at the equalities
\begin{multline*}
\int_{D}\psi^2(x)\varphi_N(x)\,d\sigma_t-
\int_{D}f^2(x, 0)\varphi_N(x)\,d\nu=
\\
=2\int_{0}^t\int_{D}|\sqrt{A}\nabla
f|^2\varphi_N\,d\sigma_s\,ds+
\\
+2\int_{0}^t\int_{D} \bigl[f(A\nabla f,
\nabla\varphi_N) +f^2L_0\varphi_N\bigr]\,d\sigma_s\,ds.
\end{multline*}
Note
that  $0\le \varphi_N(x)\le 1$.
By the maximum principle $|f(x, s)|\le \max_x|\psi(x)|\le 1$.
Using Cauchy's inequality, we obtain the following estimate:
$$
\int_{0}^t\int_{D}\varphi_N|\sqrt{A}\nabla f|^2\,d\sigma_s\,ds\le
2+\int_0^t\int_D|\sqrt{A}\nabla\varphi_N|^2\varphi_N^{-1}+|L_0
\varphi_N|\,d\sigma_s\,ds.
$$
Thus,
$$
\int_{D}\psi\varphi_N\,d\mu_t\le
(1+R_N)^{1/2}R_N^{1/2}+R_N,
$$
where
$$
R_N=\int_{0}^T\int_{D}|L_0
\varphi_N|+2\varphi_N^{-1}|\sqrt{A}\nabla\varphi_N|\,d\sigma_s\,ds.
$$
Finally,
letting $N\to\infty$, we arrive at the estimate
$$
\int_{D}\psi(x)\,d\mu_t\le 0.
$$
Replacing $\psi$ by $-\psi$, we get the opposite inequality. Hence for a.e.
$t\in[0, T]$ we have
$$
\int_{D}\psi(x)\,d\mu_t=0.
$$
Since $\psi$ was an arbitrary function in $C^{\infty}_0(\mathbb{R}^d)$ with the
only restriction $|\psi|\le 1$, we conclude that $\mu_t=0$ and hence $\sigma^1=\sigma^2$.

In the general case (without assumptions about the smoothness of coefficients)
one has to  solve the Cauchy problems
$$
\partial_sf_n+L_nf_n=0, \quad f_n|_{s=t}=\psi,
$$
in order to get an analogous estimate,  where the
coefficients of the operator
$L_n$ are smooth approximations of the coefficients of
$L$ on $\mbox{supp}\,\varphi_N$ (more precisely see Theorem 3.1 \cite{BRSH-Servy}).
\end{proof}

\begin{example}\rm
The assumption of Theorem \ref{th-uniq-1} is  fulfilled for all measures $\mu\in\mathcal{M}_{\nu}$ if
$$
(1+V(x))^{-1}|L_0V(x, t)|+(1+V(x))^{-2}|\sqrt{A}(x, t)\nabla V(x)|^2\le W(x, t),
$$
where $W\in C^{2, 1}(D\times(0, T))\bigcap C(D\times[0, T))$, $W(\, \cdot \,,0)\in L^1(\nu)$,
for every interval $[\alpha, \beta]$ from $(0, T)$ one has
$$
\lim_{k\to\infty}\inf_{D_k\setminus D_{k-1}\times[\alpha, \beta]}W(x, t)=+\infty
$$
and for some functions  $K, H\in L^1((0, T))$ with  $H\ge 0$ the estimate
$$
\partial_tW(x, t)+LW(x, t)\le K(t)+H(t)W(x, t)
$$
holds.

Indeed, by  Theorem \ref{th-est} we have $W\in L^1(\mu, D\times[0, T])$, hence
$$
\lim_{N\to\infty}\int_0^T\int_{N<V<2N}W(x, t)\,d\mu=0.
$$
\end{example}

Let us note that the assumptions of Theorem \ref{th-uniq-1} admit
practically any growth of the coefficients (the function  $W$ from the  example above
may grow arbitrarily fast), but we impose restrictions on  $|L_0V|$, i.e.,  we
control the growth of  $L_0V$  from both sides. Moreover,
we have proved the uniqueness only in the class $\mathbb{M}_{\nu}$, but {\it not} for all
measures from $\mathcal{M}_{\nu}$. It is possible to eliminate these
 constraints in the case of more regular coefficients.

We shall assume now that along with (H1) we have

\, (H2) \, for every positive integer $k$ there is a number $\Lambda_k>0$ such that
$$
|a^{ij}(x, t)-a^{ij}(y, t)|\le\Lambda_k|x-y|
$$
for all $x, y\in D_k$ and $t\in[0, T]$.

Let us recall some facts from \cite{BKR} for completeness. The assumption (H1)
ensures the existence of a density $\varrho$ of the solution $\mu$ with respect to
Lebesgue measure. Moreover, if along with  (H1) and (H2) we have $b\in
L^{p}_{loc}(D\times(0, T))$ and $c\in L^{p/2}_{loc}(D\times(0, T))$ for some
$p>d+2$, then we can choose a  version of
$\varrho$ continuous on  $D\times(0, T)$ such that
 for a.e. $t\in (0, T)$ the function $\varrho(\, \cdot \,, t)$
belongs to  $W^{1, p}(U)$ for every closed ball $U\subset D$. Since for a.e.
$t\in (0, T)$ the measure $\mu_t(dx)=\varrho(x, t)\,dx$ is a subprobability
measure on $D$,  Harnack's inequality ensures that for every closed ball $U$
from $D$ and for every interval  $J\subset(0, T)$ there exists a number
$C>0$ such  that  $\varrho(x, t)\ge C$ for all $(x, t)\in U\times J$.

Below we deal with the  continuous version of the density $\varrho$.

We recall that,  for every measure $\mu$ given by a Sobolev density $\varrho$ with respect to Lebesgue measure,
its logarithmic gradient $\beta_{\mu}$
with respect to the metric generated by the matrix $A$ is defined
by the following formula:
$$
\beta^{i}_{\mu}=\sum_{j=1}^{d}\bigl(\partial_{x_j}a^{ij}+a^{ij}\varrho^{-1}\partial_{x_j}\varrho\bigr).
$$

Further in this section we assume that the coefficients $b$ and $c$ are locally
integrable with respect to  Lebesgue measure on $D\times (0, T)$
 to power $p$ and $p/2$, respectively, for some  $p>d+2$
and that the conditions (H1) and (H2) are fulfilled.

Moreover, we consider only the set  $\mathbb{M}_{\nu}$ of measures
$\mu\in \mathcal{M}_{\nu}$ satisfying the condition
$$
b\in L^2(\mu, D_k\times[0, T]) \quad \forall D_k.
$$
For example, the latter condition is fulfilled for the whole class $\mathcal{M}_{\nu}$ in the
case of a drift bounded  on $D_k\times[0, T]$   or in
the case  $b\in L^s(D_k\times[0, T])$ and $\mu=\varrho\,dx\,dt$ with $\varrho\in
L^r(D_k\times[0, T])$, where $2/s+1/r=1$.

Suppose  that there are two solutions to the Cauchy problem (\ref{e1}) in the class $\mathbb{M}_{\nu}$
given by densities $\sigma$ and $\varrho$ with respect to Lebesgue measure.
Then these densities are continuous on $D\times(0, T)$.
In addition, the functions $\sigma$ and $\varrho$ are strictly positive.
Let $v(x, t)=\sigma(x, t)/\varrho(x, t)$. The function $v$ is continuous and positive on $D\times(0, T)$.

\begin{lemma}\label{lem4.1}
Suppose that for a.e. $t\in(0, T)$ we have the estimates
$$
\int_{D}\varrho(x, t)\,dx=\nu(D)+\int_0^t\int_{D}c(x, s)\varrho(x, s)\,dx\,ds
$$
and
$$
\int_{D}\sigma(x, t)\,dx\le\nu(D)+\int_0^t\int_{D}c(x, s)\sigma(x, s)\,dx\,ds.
$$
Suppose also that for every  $\lambda>0$ we have for a.e.  $t\in (0, T)$
\begin{equation}\label{main2}
\int_{\mathbb{R}^d}e^{\lambda(1-v(x, t))}\varrho(x, t)\,dx\le 1.
\end{equation}
 Then  $v\equiv 1$, i.e.,  $\sigma=\varrho$.
\end{lemma}
\begin{proof}
Let  $t$ be such that  $\varrho(\, \cdot \,, t)$ and $\sigma(\, \cdot \,, t)$
satisfy the listed conditions and  (\ref{main2}) holds for every positive integer
$\lambda$. We observe that the set of points  $t$ for which this is not true is a
set of zero Lebesgue measure.
If there is a ball $U\subset D$ such that
$v(x, t)\le 1-\delta$ for each $x\in U$ and some $\delta>0$, then
$$
e^{\lambda\delta}\int_{U}\varrho\,dx\le
\int_{U}e^{\lambda(1-v(x, t))}\varrho(x, t)\,dx\le 1.
$$
Letting  $\lambda\to +\infty$, we obtain a contradiction.
Hence, $v(x, t)\ge 1$ and $\sigma\ge \varrho$ for all $(x, t)\in D\times(0, T)$.
Moreover,
$$
\int_{D}\varrho(x, t)\,dx\le \int_{D}\sigma(x, t)\,dx,
\quad \int_0^t\int_{D}|c(x, s)|\varrho(x, s)\,dx\,ds\le
\int_0^t\int_{D}|c(x, s)|\sigma(x, s)\,dx\,ds.
$$
We observe that
\begin{multline*} \nu(D)=\int_{D}\varrho(x,
t)\,dx+\int_0^t\int_{D}|c(x, s)|\varrho(x, s)\,dx\,ds\le
\\
\int_{D}\sigma(x,
t)\,dx+\int_0^t\int_{D}|c(x, s)|\sigma(x, s)\,dx\,ds=\nu(D).
\end{multline*}
Hence we have
$$
\int_{D}\varrho(x, t)\,dx=\int_{D}\sigma(x, t)\,dx,
$$
which ensures  $v\equiv 1$. The lemma is proved.
\end{proof}

The next lemma is crucial in our approach.
In this lemma, it is possible to take $e^{\lambda(1-z)}$ and $e^{\lambda(1-z)}-e^{\lambda}$  for $f$.

\begin{lemma}\label{lem4.2}
Let $\psi\in C^{\infty}_0(D)$, $\psi\ge 0$ and $0<t<T$.
Then the following estimate holds:
\begin{equation}\label{in4.2.2}
\int_{D}f(v(x, t))\varrho(x, t)\psi(x)\,dx\le f(1)\int_{D}\psi(x)\,d\nu+
\int_0^t\int_{D}\varrho f(v)L\psi\,dx\,ds.
\end{equation}
If, in addition, $(b-\beta_{\mu})\varrho\in L^1(D_k\times (0, T))$
for every $k$, then
\begin{multline}\label{in4.2.3}
\int_{D}f(v(x, t))\varrho(x, t)\psi(x)\,dx\le f(1)\int_{D}\psi(x)\,d\nu+
\\
+\frac{1}{2}\int_0^t\int_{D}\varrho(A\nabla\psi, \nabla\psi)\psi^{-1}|f'(v)|^2f''(v)^{-1}\,dx\,ds+
\\
+\int_0^t\int_{D}f(v)(b-\beta_{\mu}, \nabla\psi)\varrho\,dx\,ds.
\end{multline}
\end{lemma}
\begin{proof}
The case $D=\mathbb{R}^d$ and $c=0$ was considered in \cite{SHTVP11}. Since
the estimate is local (only on  $\mbox{supp}\,\psi$), there is no difference between
arbitrary $D$ and $\mathbb{R}^d$. The addition of a new term $c$ does not give new difficulties
and the reasonings are completely the same.
\end{proof}

\begin{theorem}\label{th2-2}
Suppose that  {\rm (H1)} and {\rm (H2)} are fulfilled and
$$
b\in L^p_{loc}(D\times(0, T)), \quad c\in L^{p/2}_{loc}(D\times(0, T))
$$
 for some  $p>d+2$.
Assume also that there exists a function  $V$ such that
$V\in C^{2}(D)$ and
$$
\lim_{k\to\infty}\inf_{D_k\setminus D_{k-1}}V(x)=+\infty.
$$
Suppose that at least one of the following conditions is fulfilled:

{\rm (i)} for some measure  $\mu\in\mathbb{M}_{\nu}$
one has
$$
\lim_{N\to\infty}\int_0^T\int_{N\le V\le 2N}N^{-1}|L_0 V|+N^{-2}|\sqrt{A}\nabla V|^2\,d\mu=0,
$$

{\rm (ii)}
for some measure  $\mu\in\mathbb{M}_{\nu}$
one has
$$
\lim_{N\to\infty}\int_0^T\int_{N\le V\le 2N}N^{-1}|(b-\beta_{\mu})\nabla V|+N^{-2}|\sqrt{A}\nabla V|^2\,d\mu=0,
$$

{\rm (iii)} for some  $K>0$ and all  $(x, t)\in D\times(0, T)$ the inequality
$$
LV(x, t)\le K+KV(x)
$$
holds. Then the class $\mathbb{M}_{\nu}$ consists of at
most one element.
\end{theorem}
\begin{proof}
Let us prove (i). Let a measure
$\mu\in\mathbb{M}_{\nu}$ satisfy the condition from (i) and have
a density $\varrho$ with respect to Lebesgue measure.
By Remark \ref{m-e}, the equality (\ref{main-e}) holds for   $\mu$. Suppose
that there is yet another measure in  $\mathbb{M}_{\nu}$ given by a density
$\sigma$. Set $v=\sigma/\varrho$. Let $\psi(x)=\zeta(V(x)/N)$, where $\zeta\in
C^{\infty}_0(\mathbb{R}^d)$ is a nonnegative function such that $\zeta(x)=1$ if
$|x|\le 1$ and $\zeta(x)=0$ if $|x|>2$, and there exists a number $M>0$ such
that for all $x\in \mbox{supp}\,\zeta$   one has
$$
|\zeta(x)|\le M,
\quad |\nabla\zeta(x)|\le M, \quad |\nabla\zeta(x)|^2\zeta^{- 1}(x)\le M.
$$
Let
$f(z)=e^{\lambda(1-z)}$. It is clear that $|f(z)|\le e^{\lambda}$ if $z\ge 0$.
Using (\ref{in4.2.3}) from Lemma \ref{lem4.2} and the fact $c\le 0$, we obtain
\begin{multline*}
\int_{D}e^{\lambda(1-v(x, t))}\varrho(x,
t)\zeta(V(x)/N)\,dx\le\int_{D}\zeta(V(x)/N)\,d\nu+
\\
+e^{\lambda}M\int_0^t\int_{N\le V\le 2N}N^{-1}|L^\circ V|+N^{-2}|\sqrt{A}\nabla
V|^2\,\varrho\,dx.
\end{multline*}
Letting  $N\to +\infty$ and using Lemma \ref{lem4.1}, we obtain the required assertion.

The case(ii) can be treated similarly. Let us now prove (iii).

Suppose that there are two measures in  $\mathbb{M}_{\nu}$ given by densities
$\sigma$ and $\varrho$ with respect to Lebesgue measure. According to  Theorem
\ref{th-est}, one has  (\ref{main-e}) for both measures. Set  $v=\sigma/\varrho$.
Let $\psi(x)=\zeta(N^{-1}V(x))$, where a nonnegative function $\zeta\in
C^{\infty}_0(\mathbb{R})$ is such that  $\zeta(0)=1$, $\zeta(z)=0$ if $|z|>1$,
$0\le \zeta\le 1$ and, moreover, $\zeta'(z)\le 0$ and $\zeta''(z)\ge 0$ if $z>0$.

Let $f(z)=e^{\lambda(1-z)}-e^{\lambda}$. Then $f(z)\le 0$ and $|f(z)|\le
2e^{\lambda}$ if $z\ge 0$. Observe that
$$
f(v)\zeta'LV\le  (K+KV)f(v)\zeta'
$$
since $f(v)\zeta'\ge 0$. Using (\ref{in4.2.2}) from Lemma \ref{lem4.2}, we
obtain
\begin{multline*}
\int_{D}(e^{\lambda(1-v(x, t))}-e^{\lambda})\varrho(x,
t)\zeta(N^{-1}V(x))\,dx\le (1-e^{\lambda})\int_{D}\zeta(N^{-1}V(x))\,d\nu+
\\
+2e^{\lambda}MN^{- 1}\int_0^t\int_{V<N}(K+KV)\varrho\,dx\,ds
+\int_0^t\int_{D}\zeta(N^{- 1}V(x))f(v(x, s))c(x, s)\varrho(x, s)\,dx\,ds.
\end{multline*}
We observe that
$$
\lim_{N\to+\infty}N^{-1}\int_0^t\int_{V<N}(K+KV)\varrho\,dx\,ds=0.
$$
Indeed, let $\gamma\in (0, 1)$
and $N>\gamma^{-1}$, then
$$
N^{- 1}\int_0^t\int_{V<N}(K+KV)\varrho\,dx\,ds\le
\gamma\int_0^t\int_{V<\gamma N}\varrho\,dx\,ds
+K\int_0^t\int_{\gamma N<V<N}\varrho\,dx\,ds.
$$
Hence we have
$$
\lim_{N\to+\infty}N^{-1}\int_0^t\int_{V<N}(K+KV)\varrho\,dx\,ds\le
\gamma\int_0^t\int_{D}\varrho\,dx\,ds.
$$
Letting $\gamma\to 0$, we obtain the required observation.
Thus, letting  $N\to+\infty$, we obtain
$$
\int_{D}(e^{\lambda(1-v(x, t))}-
e^{\lambda})\varrho(x, t)\,dx\le (1-
e^{\lambda})\int_{D}\,d\nu+
\int_0^t\int_{D}(e^{\lambda(1-v)}-e^{\lambda})c(x, s)\varrho(x, s)\,dx\,ds.
$$
Since $c\le 0$ and  for a.e. $t\in (0, T)$
the identity
$$
\int_{D}\varrho(x, t)\,dx=\nu(D)+\int_0^t\int_{D}c(x, s)\varrho(x, s)\,dx\,ds,
$$
holds. Then for a.e. $t$ we have
$$
\int_{\mathbb{R}^d}e^{\lambda(1-v(x, t))}\varrho(x, t)\,dx\le 1.
$$
Using Lemma \ref{lem4.1}, we complete the proof.
\end{proof}

A combination of  Theorem \ref{th1-2} and Theorem \ref{th2-2} yields
the following sufficient conditions for existence and uniqueness.

\begin{theorem}\label{ExiUni}
Suppose that {\rm (H1)} and {\rm (H2)} hold and
$$
c\in L^{p/2}_{loc}(D\times(0, T)), \quad b\in L^{\infty}(D_k\times[0, T])
$$
for some $p>d+2$ and all $k$. Assume that there exists a function  $V$ such that
$V\in C^{2}(D)$,
$$
\lim_{k\to\infty}\inf_{D_k\setminus D_{k-1}}V(x)=+\infty,
$$
and for some number $K>0$ and all $(x, t)\in D\times(0, T)$ one has
inequality
$$
LV(x, t)\le K+KV(x).
$$
Then the class  $\mathcal{M}_{\nu}$,
where $\nu$ is a probability measure on  $D$, consists of exactly one element
$\mu=(\mu_t)_{t\in(0, T)}$. Moreover, for a.e. $t$ the identity
$$
\mu_t(D)=\nu(D)+\int_0^t\int_Dc(x, s)\,d\mu_s\,ds
$$
holds. In particular, if $c=0$, the measures $\mu_t$ are probabilities for a.e. $t$.
\end{theorem}

\begin{remark}\rm
In Remark \ref{rm1} we discussed a construction  of a
solution given by a density  $\varrho\in L^{\infty}([0, T], L^p(D))$ in the case
of a degenerate diffusion matrix. Following \cite{DL} and \cite{LeBrL}, one can
find sufficient conditions for the uniqueness of a solution. Suppose that, in
addition to the conditions from  Remark \ref{rm1},
there exists a function $V\in C^{2}(D)$ such that
$$
\lim_{n\to\infty}\inf_{D_n\setminus D_{n-1}}V(x)=+\infty
$$
and for some number $K>0$ and all  $(x, t)\in D\times(0, T)$ on has
$$
L_0V(x, t)\ge -K- KV(x), \quad |\sqrt{A(x, t)}V(x)|\le KV(x),
$$
where, as
above, $L_0 \psi=a^{ij}\partial_{x_i}\partial_{x_j}\psi+b^i\partial_{x_i}\psi$.
Then a solution of class $L^{\infty}([0, T], L^p(D))$ is unique.

Indeed, one can show that, for any solution  $\varrho$ to the Cauchy problem with zero initial
condition and  $\psi\in C^{\infty}_0(D)$,  the following equality holds:
$$
\int_D|\varrho(x, t)|^p\psi(x)\,dx\le\int_0^t\int_{D}\bigl[L_0\psi +\psi(pc+(p-
1){\rm div} h)^{+}\bigr]|\varrho|^p\,dx\,dt.
$$
Let $\psi_N(x)=\zeta(N^{-1}V(x)),$
where  $\zeta\in C^{\infty}_0(\mathbb{R})$ is a nonnegative function
such that $\zeta(z)=1$ if $|z|<1$ and $\zeta(z)=0$ if $|z|>2$, $0\le \zeta\le 1$,
moreover,  $\zeta'(z)\le 0$. We observe that for some number $C_1>0$ and all
$(x, t)\in D\times(0, T)$ we have
$$
L_0\psi(x, t)=N^{-1}\zeta'(N^{-1}V(x))L_0V(x, t)
+N^{-2}\zeta''(N^{-1}V(x))|\sqrt{A(x, t)}\nabla V(x)|^2\le K_1.
$$
Hence,
$$
\int_D|\varrho(x, t)|^p\psi_N(x)\,dx\le
K_1\int_0^t\int_{N<V<2N}|\varrho|^p\,dx\,dt+
\int_0^t\int_{D}\psi_N(pc+(p-1){\rm div} h)^{+}|\varrho|^p\,dx\,dt.
$$
Letting $N\to\infty$, we
arrive at the equality
$$
\int_D|\varrho(x, t)|^p\,dx\le\int_0^t\int_{D}(pc+(p-1){\rm div} h)^{+}|\varrho|^p\,dx\,dt.
$$
Grownwall's inequality yields that
$$
\int_D|\varrho(x, t)|^p\,dx=0
$$
and $\varrho\equiv 0$. This means exactly the
uniqueness of a solution.
\end{remark}

\begin{remark}\label{rm_erg}\rm
Let $c=0$, let $a^{ij}, b^i\in C(D)$, and let ${\rm det}A$
be nonvanishing. Assume also that a function  $V$ is such
that  $V\in C^{2}(D)$ and
$$
\lim_{n\to\infty}\inf_{D_n\setminus D_{n-1}}V(x)=+\infty.
$$
Suppose that for some number  $K>0$ and some number $n$ the
estimate  $LV(x)\le -KV$ holds for every  $x\in D\setminus D_n$.
Then, for every
probability measure $\nu$, there exists a unique solution to the Cauchy problem
$\mu=(\mu_t)_{t\in(0, +\infty)}$ given by probability measures $\mu_t$. Moreover, the
solution is ergodic, i.e.,
the measures
$$
\sigma_t(dx)=t^{-1}\int_0^t\mu_s(dx)\,ds
$$
converge weakly to a
probability solution $\mu$ of the stationary equation $L^{*}\mu=0$ on $D$ as
$t\to+\infty$.

The existence and uniqueness of a solution on  $(0, +\infty)$ follow from the
theorems above. The solution is a probability solution by the fact that $c=0$ and the
existence of a Lyapunov function. The Grownwall's inequality and a reasoning,
similar to the proof of Theorem \ref{th-est} yield that
$$
\int_{D}V(x)\,d\sigma_t=t^{- 1}\int_0^t\int_{D}V(x)\,d\mu_s\,ds\le K_1,
$$
where $K_1$ does not depend on $t$.
Hence, the family  $\sigma_{t}$ is uniformly tight
and each sequence $\sigma_{t_n}$ contains a subsequence
 weakly convergent to some measure $\sigma$.
Obviously, $\sigma$ satisfies the equation $L^{*}\sigma=0$.
By the uniqueness of a probability solution $\sigma$ the whole subsequence
$\sigma_{t_n}$ converges to it (the uniqueness for
stationary equations follows from the existence of a Lyapunov function and in
the case of an arbitrary domain $D$  can be justified in the same way as in
\cite{BRSH-Servy-e} for $D=\mathbb{R}^d$) .
\end{remark}

\begin{remark}\label{rm_KCh}\rm
Suppose that all conditions sufficient for the existence and uniqueness in the
class $\mathcal{M}_{\nu}$ are fulfilled for  $L$. Let
$\mu_{s, y}$ denote the solution of the  Cauchy problem with the initial condition
$\mu|_{t=s}=\delta_y$. We observe that for every Borel set $B$ the mapping
$(s, y)\mapsto\mu_{s, y}(B)$ is measurable as a limit of measurable mappings which
correspond to the solutions $\mu_{s, y}^n$ of the approximating problems with
smooth coefficients (in the existence theorems the solution is constructed in
exactly the same way; the fact that the constructed solution is unique ensures
convergence of the whole sequence and not only of some subsequence). Moreover,
$\mu_{s, y}^n$ satisfy the Kolmogorov--Chapman equations (see, for example,
\cite{Fr}). This ensures that  $\mu_{s, y}$ also satisfies these equations.
\end{remark}

To conclude, we shall consider some more examples. We  begin with the example
from  \cite{K-G}, which has already been mentioned in the introduction. We point out once
again that this example motivated our investigation.

\begin{example}\rm
 Let  $\nu$ be a probability measure on  $D=(-1, 1)$. Given   $\alpha>0$,
we consider the Cauchy problem
\begin{equation}\label{eq:CP-GK}
\partial_{t}\mu_{t}=
\frac{1}{2}\partial_{xx}\left(\left|1-\left|x\right|\right|^{2\alpha}\mu\right)
-\partial_{x}\left(\left({\rm tg}\left(-\frac{\pi x}{2}\right)+{\rm sign}\, x\right)\mu\right),
\quad \mu_{0}=\nu.
\end{equation}

Note that the coefficients in the above equations are rather irregular. The
drift coefficient is  discontinuous  at $x=0,\, x=2k+1$. Moreover, the diffusion
coefficient does not satisfy the linear growth condition for $\alpha>1/2$  and
is not H\"older continuous with exponent $1/2$ if $\alpha<1/2$.

Let us show that this  Cauchy problem has a unique probability solution
$\left\{\mu_{t}\right\} _{t\in[0,T]}$
in the domain  $D=(-1,1)$ for every $T>0$.

We introduce the following exhaustion $\left\{D_{k}\right\} _{k\in\mathbb{N}}$ of the domain  $D$:
$$
D_{k}=\left(-1+2^{-k},\,1-2^{-k}\right).
$$
Note that the diffusion coefficient is non-degenerate on
each cylinder $D_{k}\times[0,T]$.
Moreover, the local regularity assumptions from
Theorem \ref{ExiUni} are fulfilled. Consider the following Lyapunov function:
$$
V(x)=\frac{2-x^{2}}{1-x^{2}}.
$$
Let us show that all the assumptions of Theorem  \ref{ExiUni} are fulfilled for this Lyapunov function.

1) $V>0$ in $D$, $V\in C^{2}\left(D\right)$ and
$$
\lim\limits _{k\rightarrow\infty}\inf\limits _{D_{k}\backslash D_{k-1}}V(x)=
\lim\limits _{k\rightarrow\infty}
\frac{2-\left(1-2^{-k}\right){}^{2}}{1-\left(1-2^{-k}\right){}^{2}}=+\infty.
$$

2) The following estimate is true:
$$
LV(x)=2^{-1}\bigl|1-
|x|\bigr|^{2\alpha}V''(x)+\bigl({\rm tg}(-\pi x/2)+
{\rm sgn}x\bigr)V'(x)\leq K_{1}\cdot V(x)
$$
for some constant $K_{1}>0$
as %#?ETO CHTO?
\begin{equation}\label{eq:lim-lyap}
\lim\limits _{|x|\rightarrow1}
\frac{LV}{V}=-\infty.
\end{equation}
Hence we can apply  Theorem \ref{ExiUni}
and obtain that (\ref{eq:CP-GK})
has a unique subprobability solution.
Furthermore,  $c=0$ and thus
for a.e. $t\in[0,T]$ the solution
$\mu_{t}$ is a probability measure.

Further, the solution is ergodic in the above sense.
Indeed, due to (\ref{eq:lim-lyap}), there exist a
\emph{positive} constant $K_{2}$
and a number
$k\in\mathbb{N}$ such that on  $D\backslash D_{k}$ one has
$LV\leq-K_{2}\cdot V$.
According to Remark \ref{rm_erg}, this inequality ensures the weak
convergence of measures
$$
\sigma_{t}(dx)=\frac{1}{t}\int_{0}^{t}\mu_{s}(dx)ds
$$
to a probability measure  $\sigma$ on $D$ which solves the stationary
equation $L^{*}\sigma=0$, as $t\rightarrow+\infty$.
\end{example}

\begin{example}\rm
Let  $\nu$ be a probability measure on $D=\mathbb{R}^{d}$.
Consider the Cauchy problem
$$
\partial_{t}\mu_{t}=\partial_{x_i}\partial_{x_j}(a^{ij}\mu)-
\partial_{x_i}(b^i\mu)+c\mu, \quad \mu|_{t=0}=\nu.
$$
Suppose that  (H1) and (H2) are fulfilled and that
$$
c\in L^{p/2}_{loc}(\mathbb{R}^d\times(0, T)),
\quad b\in L^{\infty}(B(0, k)\times[0, T])
$$
for some $p>d+2$ and all $k$, where $B(0, k)$ is the ball  of radius $k$ centered at
the origin. Let  $V(x)=|x|^2/2$. Then the condition $LV\le K+KV$ takes the form
$$
{\rm tr}A(x, t)+(b(x, t), x)+|x|^2c(x, t)/2\le K+K|x|^2
$$
for all $(x,t)\in \mathbb{R}^d\times[0,T]$.
If the latter inequality holds,
then the set $\mathcal{M}_{\nu}$ consists of exactly one element.
\end{example}

The authors are grateful to V.I. Bogachev for fruitful discussions
and valuable remarks.

The work was partially supported by the RFBR grant 12-01-33009. The second
author was also partially supported by the Simons-IUM Fellowship,
the RFBR grant 11-01-00348-a and program SFB of Bielefeld University.

\end{document}